\documentclass[11pt]{amsart} 
\usepackage[utf8]{inputenc}     

\usepackage{amssymb}
\usepackage{amsthm}
\usepackage{amsmath}
\usepackage{amsxtra}
\usepackage{color}
\usepackage{mathrsfs}
\usepackage{enumerate}
\usepackage{fullpage}
\usepackage{colonequals}
\usepackage{hyperref}
\hypersetup{colorlinks=true,urlcolor=blue,citecolor=blue,linkcolor=blue}

\numberwithin{equation}{section}

\theoremstyle{plain}
\newtheorem{theorem}[equation]{Theorem}
\newtheorem{corollary}[equation]{Corollary}
\newtheorem{prop}[equation]{Proposition}
\newtheorem{lemma}[equation]{Lemma}
\newtheorem{conj}[equation]{Conjecture}

\theoremstyle{definition}
\newtheorem{dfn}[equation]{Definition}

\theoremstyle{remark}
\newtheorem{remark}[equation]{Remark}

\theoremstyle{remark}

\newenvironment{enumalph}
{\begin{enumerate}}
{\end{enumerate}}

\newenvironment{enumroman}
{\begin{enumerate}
}
{\end{enumerate}}

\newcommand{\defi}[1]{\textsf{#1}} 	

\newcommand{\F}{\mathbb F}

\newcommand{\R}{\mathbb R}
\newcommand{\Or}{\mathcal{O}}
\newcommand{\calO}{\Or}
\newcommand{\Z}{\mathbb Z}

\newcommand{\frakp}{\mathfrak{p}}
\newcommand{\p}{\mathfrak{p}}
\newcommand{\q}{\mathfrak{q}}

\newcommand{\n}{\mathfrak{n}}

\DeclareMathOperator{\Cl}{Cl}
\DeclareMathOperator{\Ma}{M}

\DeclareMathOperator{\disc}{disc}
\DeclareMathOperator{\codiff}{codiff}
\DeclareMathOperator{\Gal}{Gal}

\DeclareMathOperator{\nrd}{nrd}
\DeclareMathOperator{\trd}{trd}

\DeclareMathOperator{\opchar}{char}
\DeclareMathOperator{\ofrac}{Frac}

\DeclareMathOperator{\rad}{rad}

\DeclareMathOperator{\Id}{Id}

\DeclareMathOperator{\Clf}{Clf}
\DeclareMathOperator{\Frac}{Frac}
\DeclareMathOperator{\discrd}{discrd}

\newcommand{\sfN}{\mathsf{N}}

\begin{document}

\title{On basic and Bass quaternion orders}

\author{Sara Chari}
\address{Department of Mathematics, Dartmouth College, 6188 Kemeny Hall, Hanover, NH 03755}
\email{sara.chari.GR@dartmouth.edu}
\urladdr{\url{https://www.math.dartmouth.edu/~schari/}}

\author{Daniel Smertnig}
\address{Department of Pure Mathematics, University of Waterloo, Waterloo, ON, Canada N2L 3G1}
\email{dsmertni@uwaterloo.ca}
\urladdr{\url{https://www.math.uwaterloo.ca/~dsmertni/}}

\author{John Voight}
\address{Department of Mathematics, Dartmouth College, 6188 Kemeny Hall, Hanover, NH 03755}
\email{jvoight@gmail.com}
\urladdr{\url{http://www.math.dartmouth.edu/~jvoight/}}

\maketitle

\begin{abstract}
A quaternion order $\calO$ over a Dedekind domain $R$ is Bass if every $R$-superorder is Gorenstein, and $\calO$ is basic if it contains an integrally closed quadratic $R$-order.  In this article, we show that these conditions are equivalent in local and global settings: a quaternion order is Bass if and only if it is basic.  In particular, we show that the property of being basic is a local property of a quaternion order.
\end{abstract}

\section{Introduction}

Orders in quaternion algebras over number fields arise naturally in many contexts in algebra, number theory, and geometry---for example, in the study of modular forms and automorphic representations and as endomorphism rings of abelian varieties.  In the veritable zoo of quaternion orders, authors have distinguished those orders having favorable properties, and as a consequence there has been a certain proliferation of terminology.  In this article, we show that two important classes of orders coincide, tying up a few threads in the literature.

\subsection*{Setup}

Let $R$ be Dedekind domain and let $F$ be its field of fractions.  Let $B$ be a quaternion algebra over $F$, and let $\calO \subseteq B$ be an $R$-order.  We say that $\calO$ is \defi{Gorenstein} if its codifferent (with respect to the reduced trace pairing) is an invertible (two-sided) $\calO$-module.  Gorenstein orders were studied by Brzezinski \cite{Brzezinski:Gororder}, and they play a distinguished role in the taxonomy of quaternion orders---as Bass notes, Gorenstein rings are ubiquitous \cite{Bass:ubiquity}.  Subsequent to his work, and given the importance of the Gorenstein condition, we say $\calO$ is \defi{Bass} if every $R$-superorder $\calO' \supseteq \calO$ in $B$ is Gorenstein.  As Bass himself showed, Bass orders enjoy good structural properties, and for many applications they provide a pleasant combination of generality and graspability.  A Bass order is Gorenstein, but not always conversely.  Being Gorenstein or Bass is a local property over $R$, because invertibility is so.  

On the other hand, we say that $\calO$ is \defi{basic} if there is a (commutative) quadratic $R$-algebra $S \subseteq \calO$ such that $S$ is integrally closed in its total quotient ring $FS$.  Basic orders were first introduced by Eichler \cite{Eichler:untersuch} over $R=\Z$ (who called them \emph{primitive}), and studied more generally by Hijikata--Pizer--Shemanske \cite{HPS:orders} (among their \emph{special} orders), Brzezinski \cite{Brzezinski:onaut}, and more recently by Jun \cite{Jun}.  The embedded maximal quadratic $R$-algebra $S$ allows one to work explicitly with them, since a basic order $\calO$ is locally free over $S$ of rank $2$: for example, this facilitates the computation of the relevant quantities that arise in the trace formula \cite{HPS}.  Locally, basic orders also appear frequently: local Eichler orders are those that contain $R \times R$, and local Pizer (residually inert) orders \cite[\S 2]{Pizer} are those orders in a division quaternion algebra that contain the valuation ring of an unramified quadratic extension.  However, it is not immediate from the definition that being basic is a local property.  

\subsection*{Results}

The main result of this article is to show these two notions of Bass and basic coincide, in both local and global settings.  We first consider the local case.

\begin{theorem} \label{thm:thm1}
Let $R$ be a discrete valuation ring \textup{(}DVR\textup{)} and let $\calO$ be a quaternion $R$-order.  Then $\calO$ is Bass if and only if $\calO$ is basic.  
\end{theorem}

Theorem \ref{thm:thm1} was proven by Brzezinski \cite[Proposition 1.11]{Brzezinski:onaut} when $R$ is a complete DVR with $\opchar R \neq 2$ and having perfect residue field; the proof relies on a lengthy (but exhaustive) classification of Bass orders.  We present two (essentially self-contained) proofs that is uniform in the characteristic, one involving the manipulation of ternary quadratic forms and the second exploiting the structure of the radical.

Next, we turn to the global case.

\begin{theorem} \label{thm:thm2}
Let $R$ be a Dedekind domain whose field of fractions is a number field, and let $\calO$ be a quaternion $R$-order.  The following statements hold.
\begin{enumalph}
\item $\calO$ is basic if and only if the localization $\calO_{(\frakp)}$ is basic for all primes $\frakp$ of $R$.  
\item $\calO$ is Bass if and only if $\calO$ is basic.  
\end{enumalph}
\end{theorem}

Theorem \ref{thm:thm2}(b) over $R=\Z$ was proven by Eichler \cite[Satz 8]{Eichler:untersuch}.  For the hard direction, showing that Bass orders are basic, Eichler considered the ternary quadratic form $Q(x,y,z)=\trd^2(x,y,z) - 4\nrd(x,y,z)$ given by the quadratic discriminant on the trace zero space.
He showed that, after dividing out a suitable constant, this form represents a primitive binary quadratic form, which in turn represents infinitely many primes.
From this he was able to conclude that $Q$ represents the discriminant of a quadratic number field.
Our proof differs in that we reduce to the local situation by first proving Theorem \ref{thm:thm2}(a).

We also prove the conclusions of Theorem 1.2 in a large number of cases in which $R$ is a Dedekind domain whose field of fractions is a global function field. In this case we can always write $R = R_{(T)}$, with $T$ a non-empty (possibly infinite) set of places and $R_{(T)}$ the ring of $T$-integers (functions allowing poles at $T$).
We show that Theorem \ref{thm:thm2} holds if either
\begin{itemize}
    \item $\#T=\infty$, or
    \item $\#T<\infty$ and $B$ is $T$-indefinite (i.e., there exists $v \in T$ that is unramified in $B$).
\end{itemize}
When $\#T<\infty$ and $B$ is $T$-definite, we lack a sufficiently general local--global result on representations by ternary quadratic forms: see Section~\ref{sec:weak}. 

In fact, we show that if $\calO$ is Bass (equivalently, basic), then $\calO$ contains \emph{infinitely many} nonisomorphic quadratic $R$-algebras $S$ and moreover they can be taken to be free as $R$-modules (Corollary \ref{cor:infinitelymany}).

If $R$ is a discrete valuation ring (DVR), several equivalent characterizations of Bass orders are known \cite[Proposition 24.5.3]{Voight:quat} and this list is further extended by our results.
For the reader's convenience we give a comprehensive list.

\begin{corollary} \label{cor:local-bass}
  Let $R$ be a DVR with maximal ideal $\p$, and let $\Or$ be a quaternion $R$-order.
  Then the following are equivalent.

  \begin{enumroman}
  \item \label{lb:bass} $\Or$ is a Bass order;
  \item \label{lb:idealizer} $\Or$ and the radical idealizer $\Or^{\natural} = O_{\mathsf L}(\rad \Or) = O_{\mathsf R}(\rad \Or)$ are Gorenstein orders;
  \item \label{lb:radical} The Jacobson radical $\rad \Or$ is generated by two elements \textup{(}as left, respectively, right ideal\textup{)};
  \item \label{lb:basic} $\Or$ is a basic order;
  \item \label{lb:twogen} Every $\Or$-ideal is generated by two elements;
  \item \label{lb:directsum} Every $\Or$-lattice is isomorphic to a direct sum of $\Or$-ideals; and
  \item \label{lb:satz8} $\Or$ is not of the form $\Or = R + \p I$ with an integral $R$-lattice $I$.
  \end{enumroman}
\end{corollary}

The implications \ref*{lb:twogen}$\,\Rightarrow\,$\ref*{lb:bass}$\,\Rightarrow\,$\ref*{lb:directsum} hold more generally \cite[Section 14.5]{Voight:quat}.
The implication \ref*{lb:directsum}$\,\Rightarrow\,$\ref*{lb:twogen} holds only in specific settings; for quaternion orders it follows from work of Drozd-- Kiri\v{c}enko--Ro\u{i}ter \cite[Proposition 12.1, 12.5]{DKR}.
While we do not give another proof of this implication, we provide a direct proof for \ref*{lb:bass}$\,\Rightarrow\,$\ref*{lb:twogen}.
With the exception of statement \ref*{lb:directsum}, we therefore give a full proof of the equivalences in Corollary~\ref{cor:local-bass}.

\subsection*{Outline}

The paper is organized as follows.  After introducing the relevant background in section \ref{sec:background}, we prove Theorem \ref{thm:thm1} and Corollary~\ref{cor:local-bass} in sections \ref{sec:locbass1}--\ref{sec:variant}.  In section \ref{sec:strongapprox}, we prove Theorem \ref{thm:thm2} in the case when strong approximation applies; in section \ref{sec:weak}, we prove Theorem \ref{thm:thm2} for definite orders over $R$ the ring of integers in a number field; we then conclude the proof of Theorem \ref{thm:thm2} in general in section \ref{sec:localizationsthm2}.

\subsection*{Acknowledgements}
The authors would like to thank Wai Kiu Chan and Tom Shemanske for helpful conversations.  Voight was supported by an NSF CAREER Award (DMS-1151047) and a Simons Collaboration Grant (550029).  Smertnig was supported by the Austrian Science Fund (FWF) project J4079-N32.  Part of the research was conducted while Smertnig was visiting Dartmouth College; he would like to extend his thanks for their hospitality.

\section{Background} \label{sec:background}

In this section, we briefly review the necessary background on orders and quadratic forms.  For a general reference, see Voight \cite{Voight:quat}.

\subsection*{Properties of quaternion orders}

Let $R$ be a Dedekind domain with $\ofrac(R)=F$.  Let $B$ be a quaternion algebra over $F$ and let $\calO \subseteq B$ be an $R$-order.

\begin{dfn} We say that $\Or$ is \defi{Gorenstein} if the codifferent 
\[ \codiff(\Or) \colonequals \{\alpha \in B: \trd(\alpha \Or) \subseteq R\} \subseteq B \] 
is invertible, and we say $\Or$ is \defi{Bass} if every $R$-superorder $\Or' \supseteq \Or$ is Gorenstein. 
\end{dfn}

For more detail and further references, see Voight \cite[\S\S 24.2, 24.5]{Voight:quat}.
Being Gorenstein is a local property---$\calO$ is Gorenstein if and only if the localizations $\calO_{(\frakp)} \colonequals \calO \otimes_R R_{(\frakp)}$ are Gorenstein for all primes $\frakp$ of $R$---so it follows that Bass is also a local property.

\begin{dfn}
We say that $\calO$ is \defi{basic} if there is a (commutative) quadratic $R$-algebra $S \subseteq \calO$ such that $S$ is integrally closed in its total quotient ring $FS$.
\end{dfn}

\begin{remark}
The term \emph{primitive} is also used (in place of basic), but it is potentially confusing: we will see below that a primitive ternary quadratic form corresponds to a Gorenstein order, not a ``primitive'' order.  
\end{remark}

\subsection*{Local properties}

Now suppose $R$ is a local Dedekind domain, i.e., $R$ is a discrete valuation ring (DVR) with maximal ideal $\frakp$ and residue field $\kappa \colonequals R/\frakp$.

The \defi{Jacobson radical} of $\Or$ is the intersection of all maximal left (or equivalently right) ideals of $\Or$.  By classification, the semisimple $\kappa$-algebra $\Or/\!\rad \Or$ is one of the following \cite[24.3.1]{Voight:quat}:
\begin{itemize}
\item $\Or/\!\rad \Or$ is a quaternion algebra (equivalently, $\Or$ is maximal);
\item $\Or/\!\rad \Or \simeq \kappa \times \kappa$, and we say that $\Or$ is \defi{residually split} (or \defi{Eichler});
\item $\Or/\!\rad \Or \simeq \kappa$, and we say that $\Or$ is \defi{residually ramified}; or
\item $\Or/\!\rad \Or$ is a separable quadratic field extension of $\kappa$ and we say that $\Or$ is \defi{residually inert}.
\end{itemize}

The \defi{radical idealizer} of $\Or$ is the left order $\Or^\natural \colonequals \Or_L(\rad \Or)$.  We can create a chain of orders by iterating this process, for $\Or \subseteq \Or^\natural$, necessarily terminating in a hereditary order.

\subsection*{Ternary quadratic forms}

Still with $R$ a DVR, we review the correspondence between quaternion orders and ternary quadratic forms.

\begin{prop}[Gross--Lucianovic] \label{prop:bijOQO}
There is a discriminant-preserving bijection $\calO \leftrightarrow Q(\calO)$ between 
\begin{center}
    quaternion $R$-orders, up to isomorphism
\end{center} 
and 
\begin{center}
    nondegenerate ternary quadratic forms over $R$, up to similarly. 
\end{center} 
Moreover, an $R$-order $\calO$ is Gorenstein if and only if the corresponding ternary quadratic form $Q(\calO)$ is primitive.  
\end{prop}

Recall that a quadratic form $Q\colon M \to R$ is \defi{primitive} if $Q(M)=R$.

\begin{proof}
See Gross--Lucianovic \cite{GL:quatrings} or Voight \cite[Chapters 5, 22]{Voight:quat}.
\end{proof}

We now briefly review the construction of the bijection in Proposition \ref{prop:bijOQO}.  Since $R$ is a PID, $\calO$ is free of rank $4$ as an $\calO$-module.  A \defi{good basis} $1,i,j,k$ for an $R$-order $\Or$ is an $R$-basis with a multplication table of the form
\begin{equation} \label{eqn:ijkabc}
\begin{aligned}
i^2&=ui-bc & \qquad jk&=a\overline{i}=a(u-i)\\
j^2&=vj-ac & ki&=b\overline{j}=b(v-j)\\
k^2&=wk-ab & ij&=c\overline{k}=c(w-k)
\end{aligned}
\end{equation}
with $a,b,c,u,v,w \in R$.  Every $R$-basis of $\calO$ can be converted to a good basis in a direct manner.  With respect to this basis, for all $x,y,z \in R$ and $\alpha=xi+yj+zk \in \Or$, we have
\begin{equation} \label{eqn:trdnrdingoodbasis}
\begin{aligned}
\trd(\alpha) &=ux+vy+wz \\
\nrd(\alpha) &=
bcx^2+acy^2+abz^2+(uv-cw)xy+(uw-bv)xz+(vw-au)yz
\end{aligned}
\end{equation}

\begin{remark}
The ternary quadratic form arising from the reduced norm in \eqref{eqn:trdnrdingoodbasis} can be recognized as the adjoint form of $Q$.
\end{remark}

Associated to $\calO$ and the good basis, we attach the ternary quadratic form $Q\colon R^3 \to R$ defined by
\begin{equation} \label{eqn:Qabcuvw}
    Q(x,y,z)=ax^2+by^2+cz^2+uyz+vxz+wxy \in R[x,y,z].
\end{equation}
The similarity class of $Q$ is well-defined on the isomorphism class of $\calO$.
Conversely, given a nondegenerate ternary quadratic form $Q\colon R^3 \to R$, we associate to $Q$ its even Clifford algebra $\calO=\Clf^0(Q)$, which is a quaternion $R$-order.  Explicitly, in the standard basis $e_1, e_2, e_3$ in which $Q(xe_1+ye_2+ze_3)=Q(x,y,z)$ is as in \eqref{eqn:Qabcuvw}, then $1,i,j,k$ is a good basis for $\Clf^0(Q)$, where
$$i \colonequals e_2e_3;\quad j \colonequals e_1e_3;\quad k \colonequals e_1e_2$$ 
(in the full Clifford algebra $\Clf(Q)$) and the multiplication table \eqref{eqn:ijkabc} holds.

A change of good basis of $\calO$ induces a corresponding change of basis of $Q$, and conversely every such change of basis of $Q$ arises from a change of good basis of $\calO$.

\section{Locally Bass orders are basic} \label{sec:locbass1}

We now embark on a proof of Theorem \ref{thm:thm1}: locally, an order is Bass if and only if it is basic.  To this end, in this section and the next let $R$ be a DVR with fraction field $F \colonequals \Frac(R)$ and maximal ideal $\frakp=\pi R$.
For $x,y \in R$, we write $\pi \mid x,y$ for $\pi \mid x$ and $\pi \mid y$.

Let $B$ be a quaternion algebra over $F$ and $\calO \subseteq B$ an $R$-order.  
According to the following lemma, we could work equivalently in the completion of $R$.  

\begin{remark}
The order $\calO$ is basic (or Bass) if and only if its completion is basic (or Bass).  Indeed, invertibility and maximality can be checked in the completion. 
First, $\Or$ is basic if and only if it contains an integral quadratic subring that is maximal in its total quotient ring, so the basic property can be checked in the completion.  Similarly, $\Or$ is Bass if and only if the codifferent of every superorder is invertible, so the Bass property can be checked in the completions.  
\end{remark}

We choose a good $R$-basis $1,i,j,k$ for $\calO$ and let $Q$ be the ternary quadratic form over $R$ associated to $\calO$ with respect to this basis, as in \eqref{eqn:Qabcuvw}.

\begin{lemma}\label{lem2} 
The order $\Or$ is \emph{not} basic if and only if for every $\alpha \in \Or$ there exists $r \in R$ such that $\pi \mid \trd(\alpha-r)$ and $\pi^2 \mid \nrd(\alpha-r)$.
\end{lemma}

\begin{proof} 
Let $\alpha \in \Or$ and consider the $R$-algebra $R[\alpha]=R+R\alpha$.  Then $R[\alpha]$ fails to be integrally closed if and only if there exists $\beta \in F[\alpha]$, integral over $R$, such that $\beta \not\in R[\alpha]$; this holds if and only if there exists $r \in R$ such that $\beta=\pi^{-1}(\alpha-r)$ is integral over $R$, which is equivalent to $\trd(\beta)=\pi^{-1}\trd(\alpha-r) \in R$ and $\nrd(\beta)=\pi^{-2}\nrd(\alpha-r) \in R$, as claimed.  
\end{proof}

A slight reformulation gives a local version of the result of Eichler \cite[Satz 8]{Eichler:untersuch}.
Recall that a \defi{semi-order} $I \subseteq B$ is an integral $R$-lattice with $1 \in I$ \cite[Section 16.6]{Voight:quat}.
\defi{Basic} semi-orders are defined analogously to basic orders.

\begin{lemma} \label{lem:local-satz8}
  A semi-order $I$ is \emph{not} basic if and only if it is of the form $I=R + \p J$ with $J \subseteq B$ an integral $R$-lattice.
\end{lemma}

\begin{proof}
  If $I = R + \p J$, then $I$ is not basic by the previous lemma.
  Conversely, if $I$ is not basic, by the proof of the previous lemma, each $\alpha \in I$ is of the form $\alpha  - r = \pi \beta$ with an integral $\beta$.
  Take $J$ to be the $R$-lattice generated by all these $\beta$.
\end{proof}

As an immediate application of Lemma \ref{lem2}, we prove one implication in Theorem \ref{thm:thm1}.

\begin{prop} \label{prop:onedir}
If $\calO$ is basic, then $\calO$ is Bass.
\end{prop}

\begin{proof}
Let $\calO$ be basic.  Then every $R$-superorder $\calO'\supseteq \calO$ is also basic.  So to show that $\calO$ is Bass, we may show that $\calO$ is Gorenstein.  To do so, we prove the contrapositive.  Suppose that $\Or$ is not Gorenstein.  Then the quadratic form $Q$ associated to $\calO$ has all coefficients $a,b,c,u,v,w \in \frakp$.  From \eqref{eqn:trdnrdingoodbasis}, we see that for all $\alpha=xi+yj+zk \in \Or$ we have $\pi \mid \trd(\alpha)$ and $\pi^2 \mid \nrd(\alpha)$.  Therefore $\calO$ is not basic by Lemma \ref{lem2}.  
\end{proof}

\begin{lemma} \label{lem:onlyresram}
If $\Or$ is maximal, residually inert, or residually split, then $\Or$ is basic and Bass.
\end{lemma}

\begin{proof}
By the previous proposition it suffices to show that $\Or$ is basic.

If $\Or$ is maximal, let $K$ be a quadratic field extension of $F$ that is contained in $B$.
Then $\Or \cap K$ is an integrally closed quadratic order.
If $\Or$ is residually split, then $\Or$ contains a subring isomorphic to the integrally closed quadratic order $R \times R$ by \cite[Proposition 23.4.3]{Voight:quat}.

Finally, suppose $\Or$ is residually inert. 
Then $\Or/\!\rad \Or$ is a separable quadratic extension of $R/\p$.
If $x + \rad \Or$ with $x \in \Or$ generates $\Or/\!\rad \Or$ over $R/\p$, then $R[x]$ is the valuation ring of $F[x]$.

See also Voight \cite[24.5.2, Proposition 24.5.5]{Voight:quat}.  
\end{proof}

\begin{remark}
    One may wonder if it is also possible to embed an integrally closed quadratic order that is a \emph{domain} into a residually split (Eichler) order of level $\frakp^e$ with $e \geq 1$.  
    We restrict to the case where $2 \in R^\times$.
    For $e=1$, the element 
    $\begin{pmatrix} 0 & 1 \\ \pi & 0 \end{pmatrix}$  
    generates the valuation ring in a ramified extension.
    Suppose now $e \ge 2$ and $R[\alpha] \subseteq \Or$ is a domain with $\alpha \not \in R$.
    Without restriction $\trd(\alpha) = 0$.
    Let
    $ \alpha = \begin{pmatrix} a & b \\ c\pi^2 & -a \end{pmatrix} \in \calO$.
    Then $-\disc(\alpha)/4 = \nrd(\alpha) = -a^2 - bc\pi^e$.
    If $R[\alpha]$ is integrally closed, then either
    \begin{itemize}
        \item $v_\p(\nrd(\alpha)) = 0$  and $-\nrd(\alpha)$ represents a non-square in $R/\p$, or
        \item $v_\p(\nrd(\alpha))=1$.
    \end{itemize} 
    Since $-\nrd(\alpha)$ is clearly a square modulo $\p$, the first case is impossible.
    On the other hand, if $v_\p(a) \ge 1$, then $v_p(\nrd(\alpha)) \ge 2$.
    So $R[\alpha]$ is not integrally closed.
    This observation justifies the (more general) definition of basic orders allowing nondomains such as $R \times R$.
\end{remark}

\begin{lemma} \label{lem3} 
Suppose $\calO$ is Gorenstein with associated quadratic form $Q$ in a good basis as in \textup{\eqref{eqn:Qabcuvw}}. Suppose also that $\Or$ is not basic. Then, the following statements hold.
\begin{enumalph}
\item If $\pi \mid 2$ in $R$, then $\pi \mid u,v,w$.
\item Suppose that $\pi \mid u,v,w$.  Let $s \in \{a,b,c\}$ and suppose $\pi \mid s$.  Then $\pi^2 \mid s$.  
\end{enumalph}
\end{lemma}

\begin{proof}
For (a), to show that $\pi \mid u$, by Lemma \ref{lem2} there exists $r \in R$ such that $\pi \mid \trd(i-r)=u-2r$; since $\pi \mid 2$, we have $\pi \mid u$.  Similarly, arguing with $j,k$ we have $\pi \mid v,w$.

For (b), without loss of generality we suppose $s=a$ and $b \in R^\times$.  By Lemma \ref{lem2}, 
\[ \pi^2 \mid \nrd(k-r)=\nrd(k)-r\trd(k)+r^2=ab-rw+r^2. \] 
But $\pi \mid a,w$, so $\pi \mid r^2$.  Thus $\pi \mid r$, so $\pi^2 \mid rw,r^2$, so $\pi^2 \mid ab$; since $b \in R^\times$, we get $\pi^2 \mid a$. 
\end{proof}

\begin{lemma}\label{lem1} 
Suppose $\Or$ is Gorenstein, not basic, and residually ramified.  Then there exists a good basis of $\Or$ such that the associated quadratic form is given by $$Q(x,y,z)=ax^2+by^2+cz^2+uyz+wxy,$$ with $\pi \mid u,w$ and $\pi^2 \mid c$ and one of the following conditions holds:
\begin{enumroman}
\item $a \in R^\times$ and $\pi^2 \mid b$; or
\item $\pi^2 \mid a$ and $b \in R^\times$ and $w=0$.  
\end{enumroman}
\end{lemma}

\begin{proof} 
As explained in section 2, a change of good basis of $\Or$ corresponds to a change of basis for $Q$, so we work with the latter.  By a standard ``normal form'' argument (see e.g.\ Voight \cite[Proposition 3.10]{Voight:idenmat}), there exists a basis $e_{11},e_{12},e_{13}$ such that $Q$ becomes
\begin{equation} 
Q_1(x,y,z)=a_1x^2+b_1y^2+c_1z^2+u_1yz
\end{equation}
with $a_1,b_1,c_1,u_1 \in R$ and not all in $\frakp$, and $u_1=0$ if $2 \in R^\times$.  Let $1,i_1,j_1,k_1$ be the corresponding good basis for $\calO$.  

We modify this basis further to obtain the desired divisibility, as follows.  First, suppose that $2 \in R^\times$.  Completing the square then swapping basis vectors, we obtain the diagonal quadratic form $Q_2(x,y,z) = a_2x^2+b_2y^2+c_2z^2$ with $a \in R^\times$.  If $b_2 \in R^\times$, then $\calO$ is not residually ramified, a contradiction, so we must have $\pi \mid b_2$ and by symmetry $\pi \mid c_2$.  By Lemma \ref{lem3}, we get $\pi^2 \mid b_2,c_2$, and we are in case (i) (which becomes case (ii) after a basis swap).  

Second, suppose that $2 \not\in R^\times$, so $\pi \mid 2$.  By Lemma \ref{lem3}(a), we have $\pi \mid u_1$.  If $\pi \mid c_1$, we keep the basis unchanged and pass all subscripts $1$ to $2$.  If $\pi \mid b_1$, we take $e_{21},e_{22},e_{23} \colonequals e_{11},e_{13},e_{12}$ (swapping second and third basis elements); in this basis, we obtain the quadratic form
\begin{equation} \label{eqn:Q2}
Q_2(x,y,z)=a_2x^2+b_2y^2+c_2z^2+u_2yz 
\end{equation}
with $a_2=a_1$, $b_2=c_1$, $c_2=b_1$, and $u_2=u_1$, with $\pi \mid c_2$.  Otherwise, suppose $b_1,c_1 \in R^\times$.  Since $\calO$ is residually ramified, we have $\calO/\!\rad \calO \simeq R/\frakp$; moreover, since $i_1^2=u_1i_1-b_1c_1$.  Reducing modulo $\frakp$, we conclude that $-b_1c_1 \in (R/\frakp)^{\times 2}$, so there exists $s \in R$ such that $s_1^2 \equiv -c_1b_1^{-1} \pmod{\frakp}$.  We take the new basis $e_{21},e_{22},e_{23} \colonequals e_{11},e_{12},e_{13}+s_1e_{12}$.  In this basis, we obtain the quadratic form \eqref{eqn:Q2} where now 
\begin{center}
$a_2=a_1$, $b_2=b_1$, $c_2=c_1+su_1+s^2b_1$, and $u_2=u_1+2sb_1$.
\end{center}  
Since $\pi \mid u_1$ and $\pi \mid (c_1+s_1^2b_1)$, we have $\pi \mid c_2$.  In all cases, we have $\pi \mid c_2$.  By Lemma \ref{lem3}, we immediately upgrade to $\pi^2 \mid c_2$.  Finally, since $\Or$ is Gorenstein, either $\pi \mid a_2$ and then $b_2 \in R^\times$ and $\pi^2 \mid a_2$ as in case (ii), or we have $a_2 \in R^\times$.  

To finish, we suppose that $a_2 \in R^\times$ and we make one final change of basis to get us into case (i).  As in the previous paragraph, we have $k_2^2=-a_2b_2$, so there exists $s_2 \in R$ such that $s_2^2 \equiv -b_2a_2^{-1} \pmod{\frakp}$.  We take the new basis $e_{31},e_{32},e_{33} \colonequals e_{21},e_{22}+s_2e_{21},e_{23}$, giving the quadratic form 
$Q_3(x,y,z)=a_3x^2+b_3y^2+c_3z^2+u_3yz+w_3xy$ and
\begin{center}
$a_3=a_2$, $b_3=b_2+a_2s_2^2$, $c_3=c_2$, $u_3=u_2$, and $w_3=2a_2s_2$.
\end{center}
Now $\pi \mid b_3$ by construction, and $\pi \mid u_3,w_3$ so we upgrade to $\pi^2 \mid b_3$ and get to case (i).  
\end{proof}

We now prove Theorem \ref{thm:thm1}.  

\begin{theorem}\label{thm1} 
The order $\Or$ is Bass if and only $\calO$ is basic. 
\end{theorem}

\begin{proof}
We proved $(\Leftarrow)$ in Proposition \ref{prop:onedir}.  We prove $(\Rightarrow)$ by the contrapositive: we suppose that $\Or$ is not basic and show $\calO$ is not Bass by exhibiting a $R$-superorder $\calO'\supseteq \calO$ that is not Gorenstein.  If already $\calO$ is not Gorenstein, then we are done; so suppose that $\Or$ is Gorenstein.  

By Lemma \ref{lem:onlyresram}, we must have $\calO$ residually ramified.  Then by Lemma \ref{lem1}, there exists a good basis for $\Or$ such that the corresponding quadratic form satisfies either (i) or (ii) from that lemma.  We treat these two cases in turn. 

We begin with case (i).  We first claim that $\pi^2 \mid u$.  By Lemma \ref{lem2}, there exists $r$ such that $\pi^2 \mid \nrd(j+k-r)=ac+ab-au-rw+r^2$; since $\pi \mid b,c,u,w$ we conclude $\pi \mid r$; then $\pi^2 \mid b,c,r^2,rw$ implies $\pi^2 \mid au$, and since $a \in R^\times$ we get $\pi^2 \mid u$.  This gives us a (minimal) non-Gorenstein superorder, as follows.  Let $i' \colonequals \pi^{-1} i$ and let $\Or' \colonequals R + Ri'+Rj+Rk$.  Then $\calO' \supseteq \calO$ and $\calO'$ has the following multiplication table, with coefficients
\begin{center}
$a' \colonequals \pi a$, $b' \colonequals \pi^{-1} b$, $c' \colonequals  \pi^{-1} c$, $u' \colonequals \pi^{-1}u$, $w' \colonequals  w$
\end{center}
in $R$:
\begin{equation} 
\begin{aligned}
(i')^2 &=\pi^{-2}(ui-bc)=u'i'-b'c' & \qquad jk &=a\overline{i} = a'\overline{i'} \\
j^2 &=-ac=-a'c' & ki' &=\pi^{-1}b\overline{j}=b'\overline{j} \\
k^2 &=wk-ab=w'k-a'b' & i'j &=\pi^{-1}c\overline{k}=c'\overline{k}.
\end{aligned} 
\end{equation}
Thus $\calO'$ is an $R$-order with $Q'(x,y,z)=a'x^2+b'y^2+c'z^2+u'yz+w'xy$, all of whose coefficients are divisible by $\pi$.  We conclude $\Or'$ is not Gorenstein and so $\calO$ is not Bass.

Case (ii) follows similarly, taking instead $j' \colonequals \pi^{-1} j$ and $\Or' \colonequals R+Ri+Rj'+Rk$, with associated quadratic form $Q'(x,y,z)=a'x^2+b'y^2+c'z^2+u'yz$ satisfying $a'=\pi^{-1} a$, $b'=\pi b$, $c'=\pi^{-1} c$, $u'=u$, all of which are divisible by $\pi$. 
\end{proof}

\begin{remark}
If $\Or$ is a Gorenstein order that is neither residually split nor maximal, the \defi{radical idealizer} $\Or^{\natural}=\Or_L(\rad \Or) = \Or_R(\rad \Or)$ is the unique minimal superorder by \cite[Proposition 24.4.12]{Voight:quat}.
In the previous proof $[\Or':\Or]_\p = \p$, and hence necessarily $\Or^{\natural}=\Or'$.
We have therefore proved that if $\Or$ and $\Or^\natural$ are both Gorenstein, then $\Or$ is basic.
We come back to this in the next section, when proving Corollary~\ref{cor:local-bass}.
\end{remark}

\section{A second proof for local Bass orders being basic} \label{sec:variant}

In this section, we given a second proof of (the hard direction of) Theorem \ref{thm:thm1} using a different argument.  We retain our notation from the previous section; in particular $R$ is a discrete valuation ring with maximal ideal $\p = \pi R$.  

By classification, we see that a quaternion $R$-order $\calO$ is a local ring (has a unique maximal left [right] ideal, necessarily equal to its Jacobson radical $\rad \calO$) if and only if $\calO$ is neither maximal nor residually split.  

\begin{lemma} \label{lem-radical}
Suppose that $\Or$ is a local ring.  Let $\alpha \in \rad \Or$.  Then the following statements hold.
  \begin{enumalph}
    \item \label{lem-radical:general} We have $\pi \mid \trd(\alpha),\nrd(\alpha)$ and $\alpha^2 \in \frakp \Or$.
    \item \label{lem-radical:nonbasic} If $\Or$ is not basic, then $\pi^2 \mid \nrd(\alpha)$ and $\alpha^2 \in \frakp \rad \Or$.
    \end{enumalph}
\end{lemma}

\begin{proof}
Since $\Or/\frakp\Or$ is artinian, $(\rad \Or)/\frakp \Or$ is nilpotent, so there exists $r \in \Z_{\geq 1}$ such that $\alpha^r \equiv 0 \pmod{\frakp \calO}$.  Thus the image of $\alpha$ in the $R/\frakp$-algebra $\Or/\frakp\Or$ satisfies $x^r=0$, so its reduced characteristic polynomial must be $x^2=0$; thus $\alpha^2 \in \frakp\calO$ and $\trd(\alpha),\nrd(\alpha) \equiv 0 \pmod{\frakp}$, proving (a).  

Since $\alpha$ satisfies its reduced characteristic polynomial $f(x)=x^2-\trd(\alpha)x+\nrd(\alpha) \in R[x]$, if $\pi^2 \mid \nrd(\alpha)$, then $f(x)$ is an Eisenstein polynomial so $R[\alpha]$ is a DVR and in particular integrally closed, contradicting that $\Or$ is not basic and proving the first part of (b).  To conclude, $\nrd(\alpha) \in \frakp^2 \subseteq \frakp \rad \Or$ so 
$\alpha^2=\trd(\alpha)\alpha-\nrd(\alpha) \in \frakp \rad \Or$. 
\end{proof}

In the next two proofs we exploit the following basic fact:
Suppose $R/\frakp \cong \Or/\!\rad \Or$ (via the natural map) and let $M$ be an $\Or$-module with $M \rad\Or =0$.
Then for every $a \in A$ and $m \in M$ we can find $r \in R$ with $a - r \in \rad\Or$ and hence $am = rm$.
Specifically, we use this to see that a set $X$ that generates $M$ as $A$-module already generates it as $R$-module.

More abstractly, the $\Or$-module structure on $M$ is in fact an $\Or/\!\rad \Or$-module structure and hence the same as the $R/\frakp$-module structure coming from the inclusion $R \subseteq \Or$.
(It is always made explicit when this is used.)

We will apply the following lemma to $\Or/\pi \rad\Or$.

\begin{lemma} \label{lem-gens}
  Let $A$ be a local artinian $R$-algebra with $R/\frakp \cong A /\!\rad A$ \textup{(}via the natural map\textup{)}.
  If $y_1$, $\ldots\,$,~$y_n$ generate $\rad A$ as ideal of $A$, then they generate $A$ as $R$-algebra.
\end{lemma}

\begin{proof}
  Let $J=\rad A$.
  Since $A$ is artinian, there exists $m \ge 0$ with $J^m = 0$.
  Let $l \in [1,m]$.
  We claim that the images of $Z_l = \{y_{\nu_1} \cdots y_{\nu_l} : \nu_1, \ldots, \nu_l \in 1,\dots,n\}$ generate $J^l/J^{l+1}$ as $R$-module.
  Indeed, the elements of the form $a_0 y_{\nu_1} a_1 y_{\nu_2} a_2 \cdots a_{l-1} y_{\nu_l} a_l$ with $a_0$,~$\ldots\,$,~$a_l \in A$ generate $J^l$ as $R$-module.
  Since $R/\frakp \cong A/J$, we can write $a_\mu = r_\mu + x_\mu$ with $r_\mu \in R$ and $x_\mu \in J$.
  Thus elements of the form $r_0 y_{\nu_1} r_1 y_{\nu_2} r_2 \cdots r_{l-1} y_{\nu_l} r_l$ generate $J^l/J^{l+1}$ as $R$-module, and since $R$ acts centrally on $A$, the set $Z_l$ suffices to generate $J^l/J^{l+1}$.

  Since $A/J \cong R/\frakp$, the $R$-module $A/J$ is generated by $1+J$.
  Using the filtration $A \supseteq J \supseteq J^2 \supseteq \cdots \supseteq J^m = 0$, we see that $\{1\} \cup Z_1 \cup \cdots \cup Z_l$ generates $A$ as $R$-module.
\end{proof}

\begin{theorem} \label{thm:2gen-basic}
  Let $\Or$ be a residually ramified Bass $R$-order and suppose that $\rad \Or$ is generated by two elements as left \textup{[}right\textup{]} ideal.
  Then $\Or$ is basic.
\end{theorem}

\begin{proof}
  Suppose $\rad \Or$ is generated by two elements $\alpha_1,\alpha_2$ (as, say, left ideal, the other case being symmetric).
  Then the same is true for $\rad \Or/(\rad \Or)^2$ as $\Or$-module, and hence also as $\Or/\!\rad \Or$-module.
  Since $\Or /\!\rad \Or \cong R/\frakp$, this just means $\dim_{R/\frakp} \rad \Or/(\rad \Or)^2 \le 2$.

  Suppose now to the contrary that $\Or$ is not basic, and let $I = \pi\rad\Or$.
  Observe $I \subseteq (\rad \Or)^2$ since $\pi\Or \subseteq \rad \Or$.
  From Lemma~\ref{lem-radical} we obtain $\alpha_1^2$, $\alpha_2^2$, $(\alpha_1+\alpha_2)^2 \in I$.
  Thus $0 \equiv (\alpha_1 + \alpha_2)^2 \equiv \alpha_1 \alpha_2 + \alpha_2 \alpha_1 \pmod I$.
  By Lemma~\ref{lem-gens} the elements $\alpha_1+I$,~$\alpha_2+I$ generate $\Or/I$ as $R$-algebra.
  Since they also anticommute, we see that $\alpha_1+I$ and $\alpha_2+I$ are normal elements in $\Or/I$.
  From this it follows that $(\rad \Or)^2/I$ is generated by $\alpha_1\alpha_2 + I$ as $\Or/I$-module.

  Again using that $\alpha_1+I$ and $\alpha_2+I$ anticommute, we have $\alpha_1(\alpha_1\alpha_2) \equiv 0 \equiv \alpha_2(\alpha_1\alpha_2)  \mod I$.
  This implies that $(\rad \Or)^2/I$ is in fact generated by $\alpha_1\alpha_2+I$ as $\Or/\!\rad \Or$-module, and hence as $R/\frakp$-vector space.

  We write $\lambda(M)$ for the length of an $\Or$-module.
  Since $\dim_{R/\frakp} \rad \Or/(\rad \Or)^2 \le 2$ we have $\lambda(\Or/(\rad \Or)^2) \le 3$.

  Because $O$ is residually ramified and $\rad \Or / \pi \Or \cong I/\pi^2\Or$ we find $\lambda(I/\pi^2\Or)=3$.
  Now
  \[
    8 = \lambda(\Or / \pi^2 \Or) = \lambda(\Or / I) + \lambda(I/\pi^2\Or) = \lambda(\Or / I) + 3
  \]
 implies $\lambda(\Or / I) = 5$.
  But $\lambda((\rad \Or)^2 / I) = 1$ implies $\lambda( \Or / (\rad \Or)^2) = 4$, a contradiction.
\end{proof}

The previous theorem together with the characterization of Bass orders  \cite[Proposition 24.5.3]{Voight:quat} implies that every (residually ramified) Bass order is basic.
Alternatively, it is easy to see directly that the assumption of Theorem~\ref{thm:2gen-basic} holds for Bass orders, as the next proposition shows.

\begin{prop} \label{prop:overorders-2gen}
  If $\Or$ and $\Or^{\natural}$ are Gorenstein $R$-orders, then $\rad \Or$ is generated by two elements \textup{(}as left, respectively, right $\Or$-ideal\textup{)}.
\end{prop}

\begin{proof}
  If $\Or$ is hereditary, then $\rad \Or$ is principal \cite[Main Theorems 21.1.4 and 16.6.1]{Voight:quat}.
  If $\Or$ is Eichler, it is easily seen from an explicit description of $\Or$ that $\rad \Or$ is generated by two elements \cite[23.4.15]{Voight:quat}.
  We thus assume without restriction that $\Or$ is a local ring.

  Let $J=\rad \Or$. Then $\Or^{\natural} = (J\Or^\#)^\#$ with $\Or^\# = \Or\alpha$ for some $\alpha \in B^\times$ \cite[Proposition 24.4.12]{Voight:quat}.
  (using that $\Or$ is Gorenstein).
  Since $J\Or^\#$ is the unique maximal left [right] $\Or$-submodule of $\Or^\#$ by the proof of the same proposition, dualizing implies that there is no right [left] $\Or$-module properly between $\Or$ and $\Or^{\natural}$  \cite[Section 15.5]{Voight:quat}.
  Hence $\Or^{\natural}/\Or$ is a cylic right [left] $\Or$-module.
  So $\Or^{\natural} = \Or + \beta \Or = \Or + \Or\beta'$ with $\beta$,~$\beta' \in \Or^{\natural}$.

  Since $\Or^{\natural}$ is also Gorenstein and $O_{\mathsf L}(J) =
  \Or^{\natural}$, the ideal $J$ is invertible and hence principal \cite[Proposition 24.2.3 and Main Theorem 16.6.1]{Voight:quat}.
  So $J = \gamma \Or^{\natural} = \Or^{\natural} \gamma'$.
  Altogether $J = \gamma \Or + \gamma\beta\Or = \Or \gamma' + \Or \gamma' \beta'$.
\end{proof}

We are now in a position to give the promised comprehensive characterization of local Bass orders.

\begin{proof}[Proof of Corollary~\textup{\ref{cor:local-bass}}]
  \ref{lb:bass}$\,\Rightarrow\,$\ref{lb:idealizer} holds by definition; \ref{lb:idealizer}$\,\Rightarrow\,$\ref{lb:radical} is Proposition~\ref{prop:overorders-2gen}; \ref{lb:radical}$\,\Rightarrow\,$\ref{lb:basic} holds by Theorem~\ref{thm:2gen-basic} for residually ramified orders, in any other case $\Or$ is basic without any assumption on $\rad \Or$ by Lemma~\ref{lem:onlyresram}.
  Propositon~\ref{prop:onedir} shows \ref{lb:basic}$\,\Rightarrow\,$\ref{lb:bass}.

  \ref{lb:basic}$\,\Rightarrow\,$\ref{lb:twogen}
  Let $S$ be a maximal order of a $F$-quadratic algebra contained in $\Or$.
  Any $\Or$-ideal $I$ is an $S$-lattice of rank $2$.
  Since $S$ is local, $I$ is a free $S$-lattice of rank two.
  Thus $I$ is generated by two elements over $S$ and also over $\Or$ (as left or right ideal).
  \ref{lb:twogen}$\,\Rightarrow\,$\ref{lb:radical} is trivial.

  \ref{lb:bass}$\,\Leftrightarrow\,$\ref{lb:directsum} by \cite[Proposition 24.5.3]{Voight:quat}. The implications \ref{lb:twogen}$\,\Rightarrow\,$\ref{lb:bass}$\,\Rightarrow\,$\ref{lb:directsum} hold in large generality, whereas \ref{lb:directsum}$\,\Rightarrow\,$\ref{lb:twogen} for quaternion orders is a result of Drozd, Kiri\v{c}enko , and Ro\u{i}ter.

  \ref{lb:basic}$\,\Leftrightarrow\,$\ref{lb:satz8} by Lemma~\ref{lem:local-satz8} (a local version of a result of Eichler \cite[Satz 8]{Eichler:untersuch}).
\end{proof}

\section{Basic orders under strong approximation} \label{sec:strongapprox}

In this section, we prove Theorem \ref{thm:thm2} when strong approximation applies. We start by showing that basic is a local property, i.e., an $R$-order $\Or$ is basic if and only if its localization at every nonzero prime $\frakp$ of $R$ is basic. 

\subsection*{Setup}

Moving now from the local to the global setting, we use the following notation.  Let $F$ be a global field and let $R=R_{(T)} \subseteq F$ be the ring of $T$-integers for a nonempty finite set $T$ of places of $F$ containing the archimedean places. Let $B$ be a quaternion algebra over $F$, and let $\Or \subseteq B$ be an $R$-order.  For a prime $\p \subseteq R$, define the normalized valuation $v_\p$ with valuation ring $R_{(\p)} \subseteq F$, and similarly define $\Or_{(\p)} \colonequals \Or \otimes_R R_{(\p)} \subseteq B$.

Let $\n \subseteq R$ be a nonzero ideal.  Let $\Id_{\n}(R)$ be the group of all fractional $R$-ideals $\mathfrak{a}$ of $F$ that are coprime to $\n$ (i.e., $v_\p(\mathfrak{a})=0$ for all $\p \mid \n$). Let $P_{\n,1}(R) \leq \Id_\n(R)$ be the subgroup of principal fractional $R$-ideals $aR$ where $a \in F^\times$ satisfies
\begin{itemize}
    \item $v_\p(a-1)\geq v_\p(\mathfrak{n})$ for all $\p \mid \n$, and 
    \item $a_v>0$ for every real place $v$, where $a_v$ is the image of $a$ under $v$. 
\end{itemize}
The \defi{ray class group} with modulus $\n$ is the quotient 
\[ \Cl_{\n} R \colonequals \Id_{\n}(R)/P_{\n,1}(R). \]
By class field theory (see Tate \cite[\S 5]{Tate2010}), the maximal abelian extension $L_{\n} \supseteq F$ of conductor~$\n$, called the \defi{(narrow) ray class field} of conductor $\n$, has $\Gal(L_{\n}\,|\,F) \simeq \Cl_{\n}(R)$ under the Artin reciprocity map. We call $L_{(1)}$ the \defi{narrow Hilbert class field} of $F$.  We have a natural surjective map $\Cl_\n(R) \to \Cl R$ and $\#\Cl_\n(R)<\infty$, so $L_\n \supseteq F$ is a finite extension.

\subsection*{Building global quadratic orders}

Using discriminants, we show how we can combine local (embedded) quadratic orders to construct a candidate global quadratic order which we may try to embed in $\calO$. 
Recall that free quadratic $R$-orders are, via the discriminant, in bijection with elements $d \in R/R^{\times 2}$ that are squares in $R/4R$.

\begin{lemma}\label{lem4} 
Suppose that $\calO_{(\p)}$ is basic for all $\p$.  Then there exist infinitely many $d \in R/R^{\times 2}$, corresponding to integrally closed quadratic $R$-orders $S$ \textup{(}up to isomorphism\textup{)}, such that $S_{(\p)}$ embeds in $\Or_{(\p)}$.
\end{lemma}

\begin{proof}
For each $\p$, let $S(\p)$ be an integrally closed quadratic $R_{(\p)}$-order in $\Or_{(\p)}$ and let $d(\p) \colonequals \disc(S(\p))$.
For each $\p \mid \disc \Or$, let $e_\p \colonequals v_\p(d(\p))$. If $\p \nmid 2R$, then $e_\p \le 1$ by maximality of $S(\p)$.  Define 
\[ \mathfrak{d} \colonequals \prod_{\p \mid \discrd(\Or)} \p^{e_\p}. \] 
By the Chebotarev density theorem applied to the Hilbert class field of $F$, there exist infinitely many prime ideals $\mathfrak{q} \subseteq R$ such that $\mathfrak{q} \nmid 2\discrd(\Or)$ and $\mathfrak{d}\mathfrak{q}=d'R$ is principal.

Let 
\begin{equation} \label{eqn:tpdef}
t_\p \colonequals 
\begin{cases} 1, & \text{if $\p \nmid 2R$;}  \\
\max\{2v_{\p}(2)+1, e_\p\}, & \text{if $\p \mid 2R$}
\end{cases}
\end{equation}
and let
\begin{equation}  \label{eqn:ndef}
\n \colonequals \prod_{\p \mid 2\discrd(\Or)} \p^{t_\p}.
\end{equation}
By the Chinese Remainder Theorem, there is an element $a \in R$ such that $a \equiv d(\p) (d')^{-1} \pmod{\p^{t_\p}}$ for each $\p \mid 2\discrd(\Or)$. By the Chebotarev density theorem applied to the ray class field of $F$ of conductor $\n$, there exist infinitely many prime elements $\pi \in R$ such that $\pi \equiv a \mod \n$. 

Define $d \colonequals d'\pi$, so $dR=\mathfrak{d'}\mathfrak{q}\pi$. Then for $\p \mid 2\discrd(\Or)$, we have $d=u_\p d(\p)$, where $u_\p=d'\pi d(\p)^{-1} \equiv 1 \pmod \n$. Because $4 \mid \n$, the element $d$ is a square in $R/4R$.  Let $S$ be the (free) quadratic $R$-order of discriminant $d$.  Then $S_{(\p)}\simeq S(\p)$ for $\p \mid 2\discrd(\Or)$, which is integrally closed. For $\p \nmid 2\discrd(\Or)$, we have $S_{(\p)} \hookrightarrow \Or_{(\p)} \simeq \Ma_2(R_{(\p)})$, and $S_{(\p)}$ is integrally closed because $v_\p(d) \leq 1$. Therefore, $S_{(\p)}$ is integrally closed for each prime $\p$, so $S$ is integrally closed. Since there were infinitely many choices for primes $\mathfrak{q}$ and $\pi$, there are infinitely many choices for $S$.
\end{proof}

\subsection*{Selectivity conditions}

We must now show that we can choose $S$ in Lemma \ref{lem4} such that $S \hookrightarrow \Or$.  

\begin{lemma} \label{lem5}  
Given $\calO$, for all but finitely many integrally closed quadratic $R$-orders $S$ we have $S \hookrightarrow \calO$ if and only if $S_{(\p)} \hookrightarrow \calO_{(\p)}$ for all primes $\p$ of $R$.
\end{lemma}

\begin{proof}
Let $\mathscr{L}$ be the set of integrally closed quadratic orders $S$ (up to isomorphism) such that $S_{(\p)} \hookrightarrow\Or_{(\p)}$. We refer to Voight \cite[Main Theorem 31.1.7]{Voight:quat}:
there exists a finite extension $L \colonequals H_{GN(\Or)} \supseteq F$ with the property that $S \in \mathscr{L}$ embeds in $\Or$ whenever $K\colonequals \ofrac(S)$ is \emph{not} a subfield of $L$. As there are only finitely many subfields $K \subseteq L$, only finitely many $S \in \mathscr{S}$ will not embed in $\calO$.
\end{proof}

\begin{lemma} \label{lem7} Suppose that $B$ is $T$-indefinite, and suppose $\Or_{(\p)}$ is basic for every prime $\p$ of $R$. Then $\Or$ contains infinitely many nonisomorphic integrally closed quadratic $R$-orders.
\end{lemma}

\begin{proof}
    Suppose that $\Or_{(\p)}$ is basic for every prime $\p$ of $R$.
    Then, $\Or_{(\p)}$ contains a maximal commutative $R_{(\p)}$-order for every prime $\p$.
    By Lemma \ref{lem4}, there exist infinitely many $d \in R/R^{\times 2}$ such that the corresponding quadratic order $S_d$ is integrally closed and $(S_d)_{(\p)} \hookrightarrow \Or_{(\p)}$ for all $\p$.
    For all but finitely many such choices of $d$, we have an embedding $S_{d} \hookrightarrow \Or$.
\end{proof}

\subsection*{Proof of theorem}

With these lemmas in hand, we now prove Theorem \ref{thm:thm2} under the hypothesis that $B$ is $T$-indefinite (and $\# T < \infty$).

\begin{proof}[Proof of Theorem \textup{\ref{thm:thm2}}, $B$ is $T$-indefinite and $\#T < \infty$]
First, part (a).  If $\Or_{(\p)}$ is basic for every prime $\p$ of $R$, then $\Or$ contains an integrally closed quadratic $R$-order by Lemma \ref{lem7}.
Conversely, if $\Or$ is basic, then it contains a maximal commutative $R$-order $S$.
Then, the localization $S_{(\p)} \colonequals S \otimes R_{(\p)}$ at every prime $\p$ is a maximal $R_{(\p)}$-order in $\Or_{(\p)}$ by the local-global dictionary for lattices, so $\Or_{(\p)}$ is basic for every prime $\p$ of $R$.

  Being Bass is a local property, and local orders are basic if and only if they are Bass by Theorem~\ref{thm1}, so (b) follows from (a).
\end{proof}

This proof gives in fact a bit more.

\begin{corollary} \label{cor:inf-indefinite}
 Suppose that $B$ is $T$-indefinite and let $\Or \subseteq B$ be an $R$-order. If $\Or$ is basic, then $\calO$ contains infinitely many nonisomorphic integrally closed quadratic $R$-orders.
\end{corollary}

\begin{proof}
   Combine Theorem \ref{thm:thm2}(a) with Lemma \ref{lem7}.
\end{proof}

\section{Basic orders and definite ternary theta series} \label{sec:weak}

In this section, we finish the proof of Theorem \ref{thm:thm2} in the remaining case of a $T$-definite quaternion algebra under some hypotheses.  For this purpose, we replace the application of strong approximation with a statement on representations of ternary quadratic forms.

\subsection*{Ternary representations}

As above, let $F$ be a global field, let $T$ be a nonempty finite set of places of $F$ containing the archimedean places, and let $R=R_{(T)} \subseteq F$ be the ring of $T$-integers in $F$.  For nonzero $a \in R$, we write $\sfN(a)\colonequals \#(R/aR)$ for the \defi{absolute norm} of $a$.

\begin{conj}[Ternary representation] \label{conj:ternary}
Let $Q \colon M \to R$ be a nondegenerate ternary quadratic form over $R=R_{(T)}$ such that $Q_v$ is anisotropic for all $v \in T$.  Then there exists $c_Q \in \R_{>0}$ such that every squarefree $a \in R$ with $\sfN(a) \geq c_Q$ is represented by $Q$ if and only if $a$ is represented by the completion $Q_v$ for all places $v$ of $F$.
\end{conj}

There are number field and function field cases of Conjecture \ref{conj:ternary} to consider;  we now present results in the cases where the conjecture holds.  

\begin{theorem}[Blomer--Harcos]\label{thm3} 
When $F$ is a number field, the ternary representation conjecture \textup{(}Conjecture \textup{\ref{conj:ternary}}\textup{)} holds for $T=\{v:v \mid \infty\}$ the set of archimedean places with an ineffective constant $c_Q$.
\end{theorem}

\begin{proof}
This is almost the statement given by Blomer--Harcos \cite[Corollary 2]{BlomerHarcos}, but where it is assumed that $Q$ is positive definite: we recover the result for $Q$ definite by multiplying $Q$ by two different prime elements with appropriate signs.
\end{proof}

\begin{remark}
Using Theorem \ref{thm3}, we can show (essentially by clearing denominators by an appropriate square) that Conjecture \ref{conj:ternary} holds for all (finite sets) $T$, but we do not need this result in what follows.
\end{remark}

In the case where $F$ is a (global) function field, we know of the following partial result.

\begin{theorem}[Altu\v{g}--Tsimerman]
The ternary representation conjecture holds with an effective constant $c_Q$ when $F=\F_p(t)$ and $p \equiv 1 \pmod{4}$ and $T=\{\infty\}$.
\end{theorem}

\begin{proof}
See Altu\v{g}--Tsimerman \cite[Corollary 1.1]{AltugTsimerman}.
\end{proof}

\subsection*{Discriminants}

As above, let $B$ be a quaternion algebra over $F$ and $\calO \subseteq B$ an $R$-order.  Define the \defi{discriminant} quadratic form on $\calO$ by 
\begin{equation}
    \begin{aligned}
    \disc \colon \Or &\rightarrow R \\
    \alpha &\mapsto \trd(\alpha)^2-4\nrd(\alpha).
    \end{aligned}
\end{equation}
We define similarly $\disc_\p \colon \Or_{(\p)} \rightarrow R_{(\p)}$ for each prime $\p$. 

We will use the following two technical lemmas on discriminants.

\begin{lemma} \label{lem:fixdisc} Let $\p \subseteq R$ be prime with $S(\p)=R_{(\p)}[\alpha_\p]$ an integrally closed quadratic order. Let $f_\p \in \Z_{>0}$ be such that $\p^{2f} \mid \disc_\p(\alpha_\p)$. Then there exists a submodule $M \subseteq \Or$ such that
\begin{enumroman}
\item $\disc(\beta) \in \p^{2f}$ for all $\beta \in M$;
\item $M_{(\q)}=\Or_{(\q)}$ for $\q \neq \p$; and
\item $S{(\p)} \subseteq M_{(\p)}$.
\end{enumroman}
\end{lemma}

\begin{proof}
First, we have that $\Or_{(\p)}$ contains $S(\p)$, which is necessarily integrally closed. Then, $\Or_{(\p)} \simeq S(\p)+S(\p) \gamma$ is an $S(\p)$-module. Moreover, $\p R_{(\p)}=\pi R_{(\p)}$ is principal.  Define $M(\p) \colonequals S(\p) + S(\p)\pi^f\gamma$. For any $\beta \in M(\p)$, we have $\disc(\beta)\in \p^{2f}R_{(\p)}$. Then, we define $M \colonequals M(\p) \cap \Or \subseteq \Or$. Since $(\pi^fR_{(\p)} \cap R)_{(\q)}= R_{(\q)}$ for all $\q \neq \p$, we have $M_{(\q)}=\Or_{(\q)}$, and $(M(\p))_{(\p)}=M_{(\p)}$.  Also, $S(\p) \subseteq M(\p)$, so in particular, we have $S(\p) \subseteq M_{(\p)}$. 
\end{proof}

\begin{lemma} \label{lem6}
Suppose $\Or_{(\p)}$ is basic for all primes $\p$.  Then there exists an $R$-lattice $M \subseteq \Or$,  a totally negative $a \in R$, and for every prime $\p$ elements $\alpha_\p \in M_{(\p)}$ such that $R_{(\p)}[\alpha_\p]$ is integrally closed and the following conditions hold:
\begin{enumroman}
\item $a^{-1}\disc|_{M} \colon M \rightarrow R$ is a positive definite quadratic form;
\item $(a^{-1}\disc_\p)(\alpha_\p) \in R_{(\p)}$ is squarefree for every prime $\p$; and
\item $\disc(\alpha_\p) \in R_{(\p)}^\times$ for all but finitely many $\p$.
 \end{enumroman}
\end{lemma}

\begin{proof}
For $\p \mid 2R$, let $\alpha_\p \in \calO_\p$ be such that $v_\p(\alpha_\p)$ is minimal and let $f_\p$ be the largest integer such that $\p^{2f_\p} \mid \disc_{\p}(\alpha_{\p})$.  Similarly, for $\p \mid 2R$, let $M^{(\p)} \subseteq \calO$ be as in Lemma \ref{lem:fixdisc} with $\disc(\beta) \in \p^{2f_\p}R_{(\p)}$ for all $\beta \in M^{(\p)}$. 

Define $$\mathfrak{b} \colonequals \prod_{\p \mid 2R} \p^{2f_\p}.$$  By the Chebotarev density theorem applied to the narrow class group, there exists a prime $\mathfrak{q} \nmid 2\discrd(\Or)$ such that $\mathfrak{b} \q=aR$ is principal and $a$ is totally negative.  Since $\q \nmid \discrd(\calO)$, we have $\Or_{(\q)} \simeq \Ma_2(R_{(\q)})$, so there exists $\alpha_\q \in \Or_{(\q)}$ with $v_\q(\disc(\alpha_\q))=1$. Let $\varrho$ be a uniformizer for $R_{(\q)}$, define $M(\q) \subseteq \Or_{(\q)}$ to be the $R_{(\q)}$-suborder with basis 
\[ \begin{pmatrix}
1 & 0\\
0 & 1
\end{pmatrix}, \begin{pmatrix}
\varrho & 0\\
0 & -\varrho
\end{pmatrix},  \begin{pmatrix}
0 &1\\
\varrho & 0
\end{pmatrix},  \begin{pmatrix}
0 & -1\\
\varrho & 0
\end{pmatrix} \] 
all of whose discriminants are divisible by $\varrho$. Define $M^{(\q)} \colonequals M(\q) \cap \Or$.  Then $\disc(M^{(\q)}) \subseteq \q$. We also have that $(M^{(\q)})_\p \simeq \Or_{(\p)}$ for all $\p \neq \q$ since $\varrho \Or \subseteq M^{(\q)}$.

For the remaining primes $\p \nmid 2aR$, let $\alpha_\p \in \calO_\p$ be such that $v_\p(\alpha_\p)$ is minimal and let $M^{(\p)} \colonequals \calO$.

Define 
\[ M \colonequals \bigcap_{\p} M^{(\p)}. \]  
By construction we have $\alpha_\p \in M_{(\p)}$ for all $\p$.  Checking locally we have $a \mid \disc(\beta)$ for all $\beta \in M$. We also have that $M_{(\p)} =M^{(\p)}$ for all $\p \mid (a)$ and $M_{(\p)}=\Or_{(\p)}$ for all $\p \nmid a$.  Now, $a^{-1}\disc|_{M} \colon M \rightarrow R$ is positive definite (because $\disc$ was negative definite and $a$ was totally negative), so (i) holds.  

To conclude, we check (ii) and (iii).  Let $e_\p=v_\p(a^{-1}\disc_\p(\alpha_\p))$ for a prime $\p$.  If $\p \mid 2R$, then $\p \mid \mathfrak{b}$ so $e_\p \leq 1$ by construction (we removed the square part).  If $\p=\q$, by construction $e_\q=0$.  Otherwise, since $\calO_\p$ is basic and $\p \nmid 2aR$, we have $e_\p \leq 1$.  In particular, $e_p=0$ for all but finitely many $\p$, so (iii) holds.  
\end{proof}

\subsection*{Proof of theorem}

We give final lemma before proving the theorem.  

\begin{lemma} \label{lem8} 
Let $B$ be a $T$-definite quaternion algebra.  Suppose that Conjecture \textup{\ref{conj:ternary}} holds over $R$.  Let $\Or \subseteq B$ an $R$-order such that $\Or_{(\p)}$ is basic for every prime $\p$ of $R$.  Then $\Or$ contains infinitely many nonisomorphic integrally closed free quadratic $R$-orders.
\end{lemma}

\begin{proof}
By Lemma~\ref{lem6}, we obtain the following: an $R$-lattice $M \subseteq \Or$,  a totally negative $a \in R$, and for every prime $\p$ elements $\alpha_\p \in M_{(\p)}$ such that $R_{(\p)}[\alpha_\p]$ is integrally closed and the conditions (i)--(iii) hold.  

For each $\p$, let $d_\p \colonequals \disc(\alpha_\p)$ and  $e_\p \colonequals v_\p(d_\p)$. Define $\mathfrak{d} \colonequals \prod_{\p} \p^{e_\p}.$
Note that if $\p^e \mid aR$ then $\p^e \mid d_\p$, so $\p^e \mid \mathfrak{d}$. Therefore, $aR \mid \mathfrak{d}$.

By the Chebotarov density theorem applied to the narrow Hilbert class field, there exists a prime $\q \nmid 2\mathfrak{d}$ such that $\mathfrak{d} \q=mR$ is principal and $m$ is totally negative. In particular, $a \mid m$.
Define $t_\frakp$ as in \eqref{eqn:tpdef} and $\n$ as in \eqref{eqn:ndef}.  Applying the Chebotarov density theorem again, this time to the ray class field with conductor $\mathfrak{n}$, there exist totally positive prime elements $\pi \nmid m$ with arbitrary large absolute norm such that $\pi \equiv m^{-1} d_{\p} \pmod{\p^{t_\p}}$ for all $\p \mid 2 \mathfrak d$.  Let $d \colonequals \pi m$.  Then $a^{-1}d$ is totally positive and squarefree by construction, and there are infinitely many such choices.  

Let $d$ be such a discriminant.  We claim that $d$ is locally represented by $\disc|_{M}$.  Indeed, we have $\alpha_\p \in M_{(\p)}$ for all $\p$ by construction.  For $\p \neq \q, \pi R$, we have $d=u_\p^2 d_\p \in R_{(\frakp)}^\times$ for some $u_\p \in R_{(\p)}^\times$, so $\disc_\p(u_\p \alpha_\p)=d \in R_{(\frakp)}/R_{(\frakp)}^{\times 2}$.  For $\p=\q,\pi R$, we have $\p \nmid 2 \discrd(\Or)$, so $M_{(\p)}=\Or_{(\p)} \simeq \Ma_2(R_{(\p)})$, so in particular $\disc_\frakp \colon \calO_{(\frakp)} \to R_{(\frakp)}$ is surjective; in particular $\disc_\frakp$ represents $d$.  

Therefore $a^{-1}d$ is locally represented by $a^{-1}\disc|_{M}$.  Therefore, if the conclusion of  Conjecture~\ref{conj:ternary} holds, taking $d$ to be of sufficiently large norm, there is an element $\alpha \in M \subseteq \Or$ with $a^{-1}\disc(\alpha)=a^{-1} d$, so $\disc(\alpha)=d$. 

Finally, let $S_d \colonequals R[\alpha] \subseteq \Or$. For $\p \mid \discrd(\Or)$, we have that $(S_d)_{(\p)}=R_{(\p)}[\alpha]=R_{(\p)}[\alpha_\p]$ is maximal in its field of fractions by construction. For $\p \nmid \discrd(\Or)$, we have $v_\p(d) \leq 1$, so again $(S_d)_{(\p)}=R_{(\p)}[\alpha]$ is maximal in its field of fractions. Therefore, $S_d$ is maximal in its field of fractions and so $\Or$ is basic. 
\end{proof}

We now prove Theorem \ref{thm:thm2} in the definite case for $R$ the ring of integers of a number field.

\begin{proof}[Proof of Theorem \textup{\ref{thm:thm2}}, $B$ definite, $R$ the ring of integers of a number field]

For part (a), if $\Or_{(\p)}$ is basic for every prime $\p$ of $R$, then $\Or$ contains an integrally closed quadratic $R$-order by Lemma \ref{lem8} using Theorem \ref{thm3}.  The converse is exactly as in the proof of Theorem~\ref{thm:thm2} in the indefinite case, as given in Section~\ref{sec:strongapprox}.

  Being Bass is a local property, and local orders are basic if and only if they are Bass by Theorem~\ref{thm1}, so (b) follows from (a).
\end{proof}

\begin{corollary} \label{cor:inf-definite}
  Suppose that $B$ is $T$-definite and let $\Or \subseteq B$ be an $R$-order. If $\Or$ is basic, then $\calO$ contains infinitely many nonisomorphic integrally closed quadratic $R$-orders.
\end{corollary}

\begin{proof}
   Combine Theorem \ref{thm:thm2} in the definite case with Lemma \ref{lem8}.
\end{proof}

\section{Localizations} \label{sec:localizationsthm2}

We conclude the proof of Theorem \ref{thm:thm2} by deducing from one Dedekind domain to an arbitrary localization.  Throughout, let $R$ be a Dedekind domain with $F=\Frac R$ and let $\Or$ be a quaternion $R$-order.

\begin{lemma} \label{l:redfin}
    Let $R' \subseteq R$ be Dedekind domains such that $F\colonequals \ofrac(R)=\ofrac(R')$ is a global field.
    Let $\Or$ be an $R$-order.
    Then there is an $R'$-order $\Or' \subseteq \Or$ such that $\Or=\Or' R$ and
    \begin{itemize}
        \item $\Or'_{(\p)} = \Or_{(\p)}$ for every prime $\p$ of $R'$ with $\p R \ne R$,
        \item $\Or'_{(\p)}$ is a maximal order for every prime $\p$ of $R'$ with $\p R = R$.
    \end{itemize}
    In particular, if $\Or$ Bass, then $\Or'$ is Bass.
\end{lemma}

\begin{proof}
    Since $R$ and $R'$ are necessarily overrings of a global ring, their class groups are finite.
    It follows that there exists a multiplicative set $S \subseteq R'$ such that $R = S^{-1}R'$ \cite[Theorem 5.5]{Gabelli:dedekind}.

    Let $\alpha_1$, $\ldots\,$,~$\alpha_m$ be generators for the $R$-module $\Or$. 
    There exists (a common denominator) $d \in S$ such that 
    \begin{equation}
    d \alpha_i \alpha_j  \in R' \alpha_1 + \cdots + R' \alpha_m \quad\text{for all $i,j=1,\dots,m$}.
    \end{equation}
    This implies $(d\alpha_i)(d\alpha_j) \in R' d \alpha_1 + \cdots + R' d \alpha_m$.
    Thus $d\alpha_1$, $\ldots\,$,~$d\alpha_m$ generate an $R'$-order $\Or'' \subseteq \Or$ with $R\Or'' = \Or$.
    In particular, $\Or''_{(\p)} = \Or_{(\p)}$ for every prime $\p$ of $R'$ with $\p \cap S = \emptyset$.
    Let $\mathcal P$ be the set of prime ideals $\p$ of $R'$ with $\p \cap S \ne \emptyset$ for which $\Or''_{(\p)}$ is not maximal.  Since any $\p \in \mathcal P$ has $\p \mid \discrd(\calO'')$, the set $\mathcal P$ is finite.  
    By the local--global dictionary for lattices, there exists an $R'$-order $\Or'$ with $\Or'' \subseteq \Or'$ such that $\Or_{(\p')} = \Or_{(\p'')}$ for all $\p \not\in \mathcal P$ and $\Or'_{(\p)}$ is maximal for $\p \in \mathcal P$.
    Since $\Or_{(\p')} = \Or_{(\p'')} \subseteq \Or_{(\p)}$ for all primes $\p$ of $R$ with $\p \cap S = \emptyset$, we still have $\Or' \subseteq \Or$.
    
    Since being Bass is a local property, and at all $\frakp$ of $R$ we have either $\calO'_{(\frakp)}$ maximal or equal to $\calO_{(\frakp)}$, the order $\Or'$ is Bass.
\end{proof}

\begin{lemma} \label{l:inf-inf}
    Suppose $F = \ofrac(R)$ is a global field, and 
    let $T$ be the \textup{(}nonempty\textup{)} set of places of $F$ such that $R = R_{(T)}$.
    Suppose $\# T = \infty$.
    If $R$ is Bass, there exist infinitely many nonisomorphic quadratic $R$-orders $S$ that embed into $\Or$.
\end{lemma}

\begin{proof}
    Since $T$ is infinite, there exists a place $v \in T$ such that $B_v$ is unramified.
    Let $T'$ be a finite set of places containing $v$ and all archimedean places of $F$.
    By Lemma~\ref{l:redfin} there exists an $R_{(T')}$-order $\Or'$ such that $\Or'R=\Or$ and $\Or'$ is Bass.
    Thus $\Or'$ is locally Bass and hence locally basic by Theorem~\ref{thm:thm1}.
    Since $\Or'$ is $T$-indefinite, Lemma~\ref{lem7} implies the claim.
\end{proof}

\begin{lemma} \label{red:Rporder}
    Let $R' \subseteq R$ be Dedekind domains with $\ofrac(R') = \ofrac(R)$ a global field.
    Suppose that every $R'$-order that is Bass is basic.
    Then every $R$-order that is Bass is basic.
\end{lemma}

\begin{proof}
    As in Lemma~\ref{l:inf-inf}.
\end{proof}

\begin{theorem} \label{thm:global}
Suppose that $F = \ofrac(R)$ is a global field.
Let $T$ be the nonempty \textup{(}possibly infinite\textup{)} set of places such that $R = R_{(T)}$.
Let $\Or$ be an $R$-order.
Suppose that one of the following conditions holds.
\begin{enumroman}
    \item $F$ is a number field.
    \item $\# T < \infty$ and $B$ is $T$-indefinite.
    \item $\# T = \infty$.
\end{enumroman}
Then the following statements hold.
\begin{enumalph}
\item $\calO$ is basic if and only if the localization $\calO_{(\frakp)}$ is basic for all primes $\frakp$ of $R$.  
\item $\calO$ is Bass if and only if $\calO$ is basic.  
\end{enumalph}
\end{theorem}

\begin{proof}
    Being Bass is a local property, and local orders are basic if and only if they are Bass by Theorem~\ref{thm1}.
    Thus it suffices to show (b).
    Basic orders are Bass by Proposition~\ref{prop:onedir}, and we are left to show that an $R$-order $\Or$ that is Bass is basic.
    
    Suppose first that $\# T < \infty$.
    If $B$ is $T$-indefinite, the claim follows from Theorem~\ref{thm:thm2} in the indefinite case, as proved in Section~\ref{sec:strongapprox}.
    Suppose that $F$ is a number field and $B$ is $T$-definite.
    Let $T'$ be the set of all archimedean places of $F$.
    Then $R_{(T')}$ is the ring of integers of $F$, and the claim holds by the proof of Theorem~\ref{thm:thm2} for the definite case in Section~\ref{sec:weak} together with Theorem~\ref{thm3}.
    Lemma~\ref{red:Rporder} shows that the result also holds for $R_{(T)}$.
    
    Finally, if $\#T=\infty$, Lemma~\ref{l:inf-inf}.
\end{proof}

Theorem \ref{thm:thm2} is now immediate.

\begin{proof}[Proof of Theorem~\textup{\ref{thm:thm2}}]
    Restrict Theorem~\ref{thm:global} to the case that $F$ is a number field.
\end{proof}

Again, our arguments actually yield a little bit more.

\begin{corollary} \label{cor:infinitelymany}
If one of the conditions in Theorem~\textup{\ref{thm:global}}\textup{(i)}--\textup{(iii)} holds, and $\calO$ is basic, then $\calO$ contains infinitely many nonisomorphic integrally closed free quadratic $R$-orders.
\end{corollary}

\begin{proof}
    Corollaries~\ref{cor:inf-indefinite} and \ref{cor:inf-definite} for  $\#T < \infty$ and Lemma~\ref{l:inf-inf} for $\# T = \infty$.
\end{proof}


\begin{thebibliography}{99}

\bibitem{AltugTsimerman}
Salim Altu\v{g} and Jacob Tsimerman, \emph{Metaplectic Ramanujan conjecture over function fields with applications to quadratic forms}, Int.\ Math.\ Res.\ Not.\ (IMRN) 2014, no.\ 13, 3465--3558. 

\bibitem{Bass:ubiquity}
Hyman Bass, \emph{On the ubiquity of Gorenstein rings}, Math.~Z.\ \textbf{82} (1963), 8--28.

\bibitem{BlomerHarcos}
V. Blomer and G. Harcos, \emph{Twisted $L$-functions over number fields and Hilbert's eleventh problem}, Geom.\ Funct.\ Anal. \ \textbf{20} (2010), 1--52.

\bibitem{Brzezinski:Gororder}
J.~Brzezinski, \emph{A characterization of Gorenstein orders in quaternion algebras}, Math.\ Scand.\ \textbf{50} (1982), no.~1, 19--24.

\bibitem{Brzezinski:onaut}
J.~Brzezinski, \emph{On automorphisms of quaternion orders}, J.\ Reine Angew.\ Math.\ \textbf{403} (1990), 166--186.

\bibitem{DKR}  Ju.\,A.~Drozd, V.\,V.~Kiri\v{c}enko, A.\,V.~Ro\u{i}ter. \emph{Hereditary and Bass orders} (Russian), Izv. Akad. Nauk SSSR Ser. Mat. \textbf{31} (1967), 1415--1436.

\bibitem{Eichler:untersuch}
M.~Eichler, \emph{Untersuchungen in der Zahlentheorie der rationalen Quaternionenalgebren}, J.\ Reine Angew.\ Math.\ \textbf{174} (1936), 129--159. 

\bibitem{Gabelli:dedekind}
Stefania Gabelli, \emph{Generalized Dedekind domains}, Multiplicative ideal theory in commutative algebra, eds.\ J.\,W.\ Brewer, S.\ Glaz, W.\,J.\ Heinzer, and B.\,M.\ Olberding, Springer, New York, 2006, 189--206.

\bibitem{GL:quatrings}
Benedict H.\ Gross, Mark W.\ Lucianovic, \emph{On cubic rings and quaternion rings}, 
J.\ Number Theory \textbf{129} (2009), no.\ 6, 1468--1478. 

\bibitem{HPS}
Hiroaki Hijikata, Arnold K.~Pizer, and Thomas R.~Shemanske, \emph{The basis problem for modular forms on $\Gamma_0(N)$}, Mem.~Amer.~Math.~Soc.\ \textbf{82} (1989), no.~418.

\bibitem{HPS:orders}
H.~Hijikata, A.~Pizer, and T.~Shemanske, \emph{Orders in quaternion algebras}, J.\ Reine Angew.\ Math.\ \textbf{394} (1989), 59--106.

\bibitem{Jun}
Sungtae Jun, \emph{On the certain primitive orders}, J.\ Korean Math.\ Soc.\ \textbf{34} (1997), no.~4, 791--807.

\bibitem{Pizer}
Arnold Pizer, \emph{The action of the canonical involution on modular forms of weight 2 on $\Gamma_0(M)$}, Math.\ Ann.\ \textbf{226} (1977), 99--116.

\bibitem{Tate2010}
J. T. Tate, \emph{Global class field theory}, Algebraic number theory, 2nd. ed., J.W.S. Cassels and A. Fr\"{o}lich, eds., London Mathematical Society, London, 2010, 163--203.



\bibitem{Voight:idenmat}
John Voight, \emph{Identifying the matrix ring: algorithms for quaternion algebras and quadratic forms}, Quadratic and higher degree forms, eds.\ K.\ Alladi, M.\ Bhargava, D.\ Savitt, and P.H.\ Tiep, Developments in Math., vol.\ 31, Springer, New York, 2013, 255--298.

\bibitem{Voight:quat}
John Voight, \emph{Quaternion algebras}, v.0.9.14, July 7, 2018, \url{http://quatalg.org}.

\end{thebibliography}
\end{document}